\documentclass[twoside,leqno,10pt, A4]{amsart}
\usepackage{amsfonts}
\usepackage{amsmath}
\usepackage{amscd}
\usepackage{amssymb}
\usepackage{amsthm}
\usepackage{amsrefs}
\usepackage{latexsym}
\usepackage{mathrsfs}
\usepackage{bbm}
\usepackage{enumerate}
\usepackage{graphicx}
\usepackage{color}
\begin{document}

\newtheorem{theorem}[subsection]{Theorem}
\newtheorem{proposition}[subsection]{Proposition}
\newtheorem{lemma}[subsection]{Lemma}
\newtheorem{corollary}[subsection]{Corollary}
\newtheorem{conjecture}[subsection]{Conjecture}
\newtheorem{prop}[subsection]{Proposition}
\numberwithin{equation}{section}
\newcommand{\mr}{\ensuremath{\mathbb R}}
\newcommand{\mc}{\ensuremath{\mathbb C}}
\newcommand{\dif}{\mathrm{d}}
\newcommand{\intz}{\mathbb{Z}}
\newcommand{\ratq}{\mathbb{Q}}
\newcommand{\natn}{\mathbb{N}}
\newcommand{\comc}{\mathbb{C}}
\newcommand{\rear}{\mathbb{R}}
\newcommand{\prip}{\mathbb{P}}
\newcommand{\uph}{\mathbb{H}}
\newcommand{\fief}{\mathbb{F}}
\newcommand{\majorarc}{\mathfrak{M}}
\newcommand{\minorarc}{\mathfrak{m}}
\newcommand{\sings}{\mathfrak{S}}
\newcommand{\fA}{\ensuremath{\mathfrak A}}
\newcommand{\mn}{\ensuremath{\mathbb N}}
\newcommand{\mq}{\ensuremath{\mathbb Q}}
\newcommand{\half}{\tfrac{1}{2}}
\newcommand{\f}{f\times \chi}
\newcommand{\summ}{\mathop{{\sum}^{\star}}}
\newcommand{\chiq}{\chi \bmod q}
\newcommand{\chidb}{\chi \bmod db}
\newcommand{\chid}{\chi \bmod d}
\newcommand{\sym}{\text{sym}^2}
\newcommand{\hhalf}{\tfrac{1}{2}}
\newcommand{\sumstar}{\sideset{}{^*}\sum}
\newcommand{\sumprime}{\sideset{}{'}\sum}
\newcommand{\sumprimeprime}{\sideset{}{''}\sum}
\newcommand{\sumflat}{\sideset{}{^\flat}\sum}
\newcommand{\shortmod}{\ensuremath{\negthickspace \negthickspace \negthickspace \pmod}}
\newcommand{\V}{V\left(\frac{nm}{q^2}\right)}
\newcommand{\sumi}{\mathop{{\sum}^{\dagger}}}
\newcommand{\mz}{\ensuremath{\mathbb Z}}
\newcommand{\leg}[2]{\left(\frac{#1}{#2}\right)}
\newcommand{\muK}{\mu_{\omega}}
\newcommand{\thalf}{\tfrac12}
\newcommand{\lp}{\left(}
\newcommand{\rp}{\right)}
\newcommand{\Lam}{\Lambda_{[i]}}
\newcommand{\lam}{\lambda}
\def\L{\fracwithdelims}
\def\om{\omega}
\def\pbar{\overline{\psi}}
\def\phis{\varphi^*}
\def\lam{\lambda}
\def\lbar{\overline{\lambda}}
\newcommand\Sum{\Cal S}
\def\Lam{\Lambda}
\newcommand{\sumtt}{\underset{(d,2)=1}{{\sum}^*}}
\newcommand{\sumt}{\underset{(d,2)=1}{\sum \nolimits^{*}} \widetilde w\left( \frac dX \right) }

\newcommand{\hf}{\tfrac{1}{2}}
\newcommand{\af}{\mathfrak{a}}
\newcommand{\Wf}{\mathcal{W}}

\theoremstyle{plain}
\newtheorem{conj}{Conjecture}
\newtheorem{remark}[subsection]{Remark}

\makeatletter
\def\widebreve{\mathpalette\wide@breve}
\def\wide@breve#1#2{\sbox\z@{$#1#2$}%
     \mathop{\vbox{\m@th\ialign{##\crcr
\kern0.08em\brevefill#1{0.8\wd\z@}\crcr\noalign{\nointerlineskip}%
                    $\hss#1#2\hss$\crcr}}}\limits}
\def\brevefill#1#2{$\m@th\sbox\tw@{$#1($}%
  \hss\resizebox{#2}{\wd\tw@}{\rotatebox[origin=c]{90}{\upshape(}}\hss$}
\makeatletter

\title[Upper bounds for shifted moments of quadratic Dirichlet $L$-functions over function fields]{Upper bounds for shifted moments of quadratic Dirichlet $L$-functions over function fields}

\author[P. Gao]{Peng Gao}
\address{School of Mathematical Sciences, Beihang University, Beijing 100191, China}
\email{penggao@buaa.edu.cn}

\author[L. Zhao]{Liangyi Zhao}
\address{School of Mathematics and Statistics, University of New South Wales, Sydney NSW 2052, Australia}
\email{l.zhao@unsw.edu.au}

\begin{abstract}
 We establish bounds on shifted moments of families quadratic Dirichlet $L$-functions over function fields.  These estimates are consistent with their conjectured orders of magnitude.  As an application, we prove some bounds for moments of quadratic Dirichlet character sums over function fields.
\end{abstract}

\maketitle



\section{Introduction}
\label{sec 1}

Moments of families of $L$-functions are studied extensively in number theory.  K. Soundararajan \cite{Sound2009} established, assuming the Riemann hypothesis, upper bounds for the $2k$-th moment, with $k>0$, of the Riemann zeta-function.  These bounds were made sharp by A. J. Harper \cite{Harper} using ideas from \cites{Sound2009, Radzi2011}.  For central values of the family of quadratic Dirichlet $L$-functions over $\mq$, asymptotic formulas were given by M. Jutila \cite{Jutila} for the first and the second moments, and by K. Soundararajan \cite{Sound2009} for the third moment. In \cite{Shen}, Q. Shen evaluated the fourth moment asymptotically under the generalized Riemann hypothesis(GRH). An unconditional result was later obtained by Q. Shen and J. Stucky \cite{ShenStucky} using ideas of X. Li \cite{Li2024} in the asymptotic evaluation of the second moment of quadratic twists of modular $L$-functions unconditionally and improving an earlier result, conditional on GRH, of K. Soundararajan and M. P. Young \cite{S&Y}.  Precise conjectural formulations of the asymptotic behaviors of all $k$-th, with $k \in [0, \infty)$, moments are given by J. P. Keating and N. C. Snaith \cite{Keating-Snaith02} and by J. B. Conrey, D. W. Farmer, J. P. Keating, M. O. Rubinstein and N. C. Snaith \cite{CFKRS}.  Related to this, sharp lower (resp. upper) bounds for these $k$-th moments have been established for all $k \geq 0$ (resp. $0 \leq k \leq 2$) unconditionally, by incorporating the results in \cites{Harper, Sound2009, R&Sound, R&Sound1, Radziwill&Sound1, Gao2021-2, Gao2021-3}. Moreover,  sharp upper bounds (see \cite{Harper}) have been established for all $k > 2$ under GRH. \newline

Instead of only focusing on the central values, the result in \cite{ShenStucky} in fact evaluates the shifted fourth moment of the underlying family of $L$-functions.  The study of shifted moments was first conducted in the work of V. Chandee \cite{Chandee} with subsequent improvement by N. Ng, Q. Shen and P.-J. Wong \cite{NgSheWon} and sharpening by M. J. Curran \cite{Curran}.  Here we note that similiar upper bounds were obtained using similar approaches by B. Szab\'o \cite{Szab} and M. J. Curran \cite{Curran} for shifted moments of the family of Dirichlet $L$-functions to a fixed modulus and of the Riemann zeta function, respectively. An extension of Szab\'o's result to the function fields was given in \cite{B&Gao2024}. \newline

In \cite{G&Zhao2024-3}, the authors applied the method in \cites{Sound2009, Harper} to show that, for fixed integer $k\geq 1$, positive real numbers $a_1,\ldots, a_{k}, A$ and a real $k$-tuple $t=(t_1,\ldots ,t_{k})$ with $|t_j|\leq  X^A$ for a large real number $X$, then, under GRH,
\begin{align*}
\begin{split}
  & \sumstar_{\substack{(d,2)=1 \\ d \leq X}}\big| L\big( \tfrac12+it_1,\chi^{(8d)} \big) \big|^{a_1} \cdots \big| L\big(\tfrac12+it_{k},\chi^{(8d)}  \big) \big|^{a_{k}} \\
\ll & X(\log X)^{(a_1^2+\cdots +a_{k}^2)/4} \\
& \times \prod_{1\leq j<l \leq k} \Big|\zeta \Big( 1+i(t_j-t_l)+\frac 1{\log X} \Big) \Big|^{a_ia_j/2} \Big|\zeta \Big(1+i(t_j+t_l)+\frac 1{\log X} \Big) \Big|^{a_ia_j/2}\prod_{1\leq j\leq k} \Big|\zeta \Big(1+2it_j+\frac 1{\log X} \Big) \Big|^{a^2_i/4+a_i/2},
\end{split}
\end{align*}
 where $\zeta(s)$ is the Riemann zeta function defined in the rational number field. \newline

The aim of this paper is to extend the above result to the rational function fields setting, in which some earlier work has been done by P. Darbar and G. Maiti \cite{DarMai}.  This will then be applied to bound the moments of quadratic character sums as in \cite{G&Zhao2024-3}.  A brief preface is required before we can state our results. \newline
 
We fix a finite field $\mathbb{F}_{q}$ of odd cardinality $q$ and denote by $A=\mathbb{F}_{q}[T]$ the polynomial ring over $\mathbb{F}_{q}$.  Let $\deg(f)$ be the degree of $f \in A$, and define the norm $|f|$ to be $|f|=q^{\deg(f)}$ for $f\neq 0$ and $|f|=0$ for $f=0$.  We set
\begin{align*}
\mathcal{H}_{2g+1,q}=& \{\text{$D\in A$ : square-free, monic and $\deg(D)=2g+1$}\}.
\end{align*}

We reserve the symbol $P$ for a monic, irreducible polynomial in $A$ and refer to $P$ as a prime in $A$. For any non-zero monic $D \in A$, the quadratic symbol $\leg {\cdot}{P}$ for any prime $P \in A$ is defined by
\begin{equation*}
\leg {f}{P}=\Bigg\{
\begin{array}{cl}
0, & \mathrm{if}\ P \mid f,\\
1, & \mathrm{if}\ P \nmid f \ \mathrm{and} \ f \ \mathrm{is \ a \ square \ modulo} \ P,\\
-1, & \mathrm{if}\ P\nmid f \ \mathrm{and} \ f \ \mathrm{is \ not \ a \ square \ modulo} \ P.\\
\end{array}
\end{equation*}
The above definition is then extended multiplicatively to $\leg {\cdot}{D}$ for any non-zero monic $D \in A$.  The quadratic reciprocity law (see \cite[Theorem 3.3]{Rosen02}) asserts that any finite field $\mathbb{F}_{q}$ of odd cardinality $q$ and non-zero monic $C$, $D \in A$ with $(C, D)=1$, 
\begin{align}
\label{quadrecp}
 \leg {C}{D}  = \leg{D}{C}(-1)^{\deg(C)\deg(D)(q-1)/2 }=\leg{(-1)^{\deg(
 D)}D}{C},
\end{align}
 where the last equality above follows from the relation (see \cite[Proposition 3.2]{Rosen02}) that for any element $a \in \mathbb{F}_{q}$,
\begin{align*}
 \leg {a}{C}  = a^{\deg(C)(q-1)/2 }.
\end{align*}
  
  We write $\chi_{D}$ for the quadratic symbol $\leg {D}{\cdot}$ throughout the paper. The $L$-function associated with $\chi_D$ for $\Re(s)>1$ is then defined as
\begin{equation*}
L(s,\chi_D)=\sum_{\substack{f\in A}}\frac{\chi_D(f)}{|f|^{s}}=\prod_{P}(1-\chi(P)|P|^{-s})^{-1}.
\end{equation*}
  Moreover, let $\zeta_A(s)$ denote the zeta function associated to $A$. \newline

We shall often write $L(s,\chi_D)=\mathcal{L}(u,\chi_D)$ via a change of variables $u=q^{-s}$, where
\begin{align}
\label{Ludef}
  \mathcal{L}(u,\chi_D) = \sum_{f \in A} \chi_D(f) u^{\deg(f)} = \prod_P (1-\chi_D(P) u^{\deg(P)})^{-1}.
\end{align}
Here and after $\prod_P$ denotes the product over all primes in $A$. \newline

Our first result bounds from above the shifted moments of $L(s, \chi_D)$ for $D \in \mathcal{H}_{2g+1,q}$, analogous to \cite[Theorem 1.1]{G&Zhao2024-3}.
\begin{theorem}
\label{t1}
With the notation as above, suppose that $X=q^{2g+1}$. Let $k\geq 1$ be a fixed integer, $a_1,\ldots, a_{k}$ fixed positive real numbers, and $t=(t_1,\ldots ,t_{k})$ be a real $k$-tuple. Then we have for $g$ larger than some constant depending on the $a_j$,
\begin{align}
\label{Lprodbounds}
\begin{split}
 \sum_{D\in\mathcal{H}_{2g+1,q}} & \big | L\big(\tfrac12+it_1,\chi_D \big)\big |^{a_1} \cdots \big |L\big(\tfrac12+it_{k},\chi_D \big)\big |^{a_{k}} \\
& \ll X(\log X)^{(a_1^2+\cdots +a_{k}^2)/4} \prod_{1\leq j<l \leq k} \Big|\zeta_A \Big(1+i(t_j-t_l)+\frac 1{\log X}\Big) \Big|^{a_ja_l/2}\Big|\zeta_A \Big(1+i(t_j+t_l)+\frac 1{\log X} \Big) \Big|^{a_ja_l/2} \\
& \hspace*{2cm} \times \prod_{1\leq j\leq k} \Big|\zeta_A \Big(1+2it_j+\frac 1{\log X}\Big) \Big|^{a^2_j/4+a_j/2}.
\end{split}
\end{align}
Consequently,
\begin{align*}
\begin{split}
\sum_{D\in\mathcal{H}_{2g+1,q}} & \big |L\big( \tfrac12+it_1,\chi_D \big)\big |^{a_1} \cdots \big |L\big( \tfrac12+it_{k},\chi_D \big)\big |^{a_{k}} \\
& \ll X(\log X)^{(a_1^2+\cdots +a_{k}^2)/4}  \prod_{1\leq j<l \leq k} \min \Big(\log X, \frac {1}{\overline{\log q |t_j-t_l|}} \Big)^{a_ja_l/2}\min \Big(\log X, \frac {1}{\overline{\log q |t_j+t_l|}} \Big)^{a_ja_l/2} \\
& \hspace*{2cm} \times \prod_{1\leq j\leq k} \min \Big( \log X, \frac {1}{\overline{\log q |2t_j|}} \Big)^{a^2_j/4+a_j/2},
\end{split}
\end{align*}
   where $\overline {\theta}=\min_{n \in \mz} |\theta-2\pi n|$ for any $\theta \in \rear$.  The implied constants above depend on $k$ and $a_j$, but not on $X$ or the $t_j$'s.
\end{theorem}

The special case of $t_j = t_l$ for all $j$, $l$ and $q$ prime was established in \cite[Theorem 2.7]{Florea17-2} with an extra factor $q^{\varepsilon}$ on the right-hand side of \eqref{Lprodbounds}. Moreover, the case of $t_j=0$ and $q$ prime is \cite[Theorem 1.1]{G&Zhao12}.  See also \cites{Andrade&Keating12, Jung, Florea17, Florea17-1, Florea17-2, Diaconu19} for asymptotic evaluations of the $k$-th moment of the family of $L$-functions studied in Theorem \ref{t1} at the central point for $k=1,2,3,4$. \newline

    Note that Theorem \ref{t1} is valid unconditionally as the Riemann hypothesis is true in the function fields setting. Also, our result depends on $\overline{\log q |t_i-t_j|}$ instead of $|t_i-t_j|$ since the $L$-functions $L(\tfrac12+it, \chi_D)$ are periodic functions of $t$. In particular, Theorem \ref{t1} is valid without any restrictions on $t_j$. \newline

  We point out here that in \cite{BerDiaPetWes},  J. Bergström, A. Diaconu, D. Petersen and C. Westerland used a novel approach from algebraic topology to unconditionally establish asymptotic formulas for all moments of central values of quadratic Dirichlet $L$-functions over function fields. One may be able to carry the treatment in \cite{BerDiaPetWes} to the case on shifted moments as well.  Nevertheless, our proof of Theorem \ref{t1} is more direct and simpler. \newline

  We also note that the Density Conjecture of N. Katz and P. Sarnak \cites{KS1, K&S} associates a classical compact group to
 each natural family of $L$-functions. Under this philosophy, the family of quadratic Dirichlet $L$-functions is a symplectic family
 of $L$-functions, whilst the family of Dirichlet $L$-functions to a fixed modulus and the family of the Riemann zeta function on the critical line are all
 unitary families of $L$-functions. Consequently, the upper bounds obtained in \eqref{Lprodbounds} differs from those obtained in \cites{Szab, Curran, B&Gao2024} in the sense that there are extra terms involving with products of various powers of
\begin{align*}
\begin{split}
 \Big|\zeta_A \Big(1+i(t_j+t_l)+\frac 1{\log X} \Big) \Big| \quad  \text{and} \quad  \Big|\zeta_A \Big(1+2it_j+\frac 1{\log X}\Big) \Big|, \ 1 \leq j, l \leq k.
\end{split}
\end{align*}

   The appearance of these terms are due to the contribution of the prime squares (see Proposition \ref{prop1} below)  as $\chi_D(P^2) = 1$ for most primes $P \in A$. In the calculations for zeta or Dirichlet $L$-functions however, the prime squares emerge quite differently. For example, in the situation of the zeta function, corresponding to the contribution from $\chi_D(P^2)$ in the symplectic case,  one need to consider averages over primes $p$ of $p^{2it}$ for a fixed $t \in \mr$. Unlike $\chi_D(P^2)$, these terms oscillate rapidly.  Hence their contributions are negligible. \newline
 
 We shall use Theorem \ref{t1} to estimate moments of quadratic Dirichlet character sums. Define, for any real $m \geq 0$,
\begin{align*}
  S_m(q, g,Y) =\sum_{D\in\mathcal{H}_{2g+1,q}} \bigg|\sum_{|f|\leq Y} \chi_D(f)\bigg|^{2m}.
\end{align*}
Here and throughout the paper, we use the convention that when considering a sum over some subset $S$ of $A$, the symbol $\sum_{f \in S}$ stands for a sum over monic $f \in S$, unless otherwise specified. \newline

   Applying arguments similar to those used in the proof of \cite[Theorem 3]{Szab}, we obtain the following bounds on $S_m(q, g, Y)$.
\begin{theorem}
\label{quadraticmean}
With the notation as above, any real $m \geq 3/2$, we have, for $q$ larger than some constant depending on $m$ and $Y \geq 1$,
\begin{align*}
 S_m(q,g,Y) \ll XY^m(\log X)^{2m^2-m+1},
\end{align*}
 where the implied constant depending on $m$ only.
\end{theorem}

\section{Preliminaries}
\label{sec 2}

\subsection{Backgrounds on function fields}
\label{sec 2.1}

We cite a few facts about function fields in this section.  Most of these can be found in \cite{Rosen02}. \newline

For $\Re(s)>1$,
\begin{equation*}
\zeta_A(s)=\sum_{\substack{f\in A}}\frac{1}{|f|^{s}}=\prod_{P}(1-|P|^{-s})^{-1},
\end{equation*}
  As there are $q^n$ monic polynomials of degree
$n$, it follows that
\begin{equation*}
\zeta_A(s)=\frac{1}{1-q^{1-s}}.
\end{equation*}
 This extends $\zeta_A(s)$ to be defined on the entire complex plane with a simple pole at $s = 1$. Also, as a special case of \eqref{Ludef}, we have
$\zeta_{A}(s)=\mathcal{Z}(u)$ with $u=q^{-s}$ and
\begin{align}
\label{Zudef}
\mathcal{Z}(u)=\prod_{P}(1-u^{\deg(P)})^{-1}=(1-qu)^{-1}.
\end{align}

Recall that $\chi_{D}$ stands for the quadratic symbol $\leg {D}{\cdot}$.  It is shown in \cite[Section 2.1]{Florea17-2} that the associated $L$-function $\mathcal{L}(u,\chi_D)$ is a polynomial in $u$ of degree at most $\deg(D)-1$ when $D$ is monic and $D \neq \square$,  where $\square$ stands for a perfect square in $A$.  Hence $\mathcal{L}(u,\chi_D)$ is defined on the entire complex plane.  Moreover, from \cite[(2.2)]{Florea17-2}, for $D \in \mathcal{H}_{2g+1,q}$,
\begin{align}
\label{LchiDdef}
 \mathcal{L}(u,\chi_D) = \prod_{j=1}^{2g} \left( 1-\alpha_j q^{1/2}u \right) .
\end{align}
 where $|\alpha_j|=1$, using the Riemann hypothesis in functions fields proved by A. Weil \cite{Weil}.

\subsection{Estimates of various sums} We begin with some formulas from \cite[Lemma 2.2]{G&Zhao12}.
\begin{lemma}
\label{RS}
   For $x \geq 2$ and some constant $b$, we have
\begin{align}
\label{logp}
\sum_{|P| \le x} \frac {\log |P|}{|P|} =& \log x + O(1),  \; \mbox{and} \\
\label{lam2p}
\sum_{|P| \le x} \frac{1}{|P|} =& \log \log x + b+ O\left( \frac {1}{\log x} \right).
\end{align}
\end{lemma}

  The next result is established in \cite[Lemma 2.4]{B&Gao2024}, which generalizes \eqref{lam2p} and is analogous to \cite[Lemma 3.2]{Kou}, \cite[Lemma 2.6]{Curran} and \cite[Lemma 2]{Szab} in the number fields setting.
\begin{lemma}
\label{mertenstype}
Let $\alpha>0$.  Then for $x \geq 2$,
\begin{align}
\label{sumcosp}
\begin{split}
\sum_{|P|\leq x} \frac{\cos(\alpha \log |P|) }{|P|}=& \log \Big| \zeta_A \Big( 1+ \frac{1}{\log x} +i\alpha \Big) \Big|+O(1) = \log \min \Big( \log x, \frac {1}{\overline {\alpha \log q}} \Big)+O(1).
\end{split}
\end{align}
\end{lemma}

  We define for any $f \in A$,
\begin{align*}
\begin{split}
  I(f)=\prod_{\substack{P|f }}\Big(1-\frac {1}{2|P|}\Big).
\end{split}
\end{align*}
As usual, throughout the paper that any empty product is defined to be $1$ and the empty sum $0$.  We observe that $E(f)>0$ for any $f$. \newline

   Now we have the following result for quadratic character sums.
\begin{lemma}
\label{prsum}
 With the notation as above and $X=q^{2g+1}$, we have, for any monic $f \in A$,
\begin{equation} \label{sumAd}
\begin{split}
\sum_{\substack{D\in\mathcal{H}_{2g+1,q} }} & I(D)^{-k}\left(\frac{D}{f}\right) \\
=& \delta_{f=\square} X\prod_{P}(1-|P|^{-1})(1+I(P)^{-k}|P|^{-1})\times \prod_{P|f}(1+
I(P)^{-k}|P|^{-1})^{-1}+
O_k(X^{1/2+\varepsilon}|f|^{\varepsilon}), \\
\end{split}
\end{equation}
and
\begin{equation} \label{meancharsum}
\sum_{\substack{D\in\mathcal{H}_{2g+1,q} }}\left(\frac{D}{f}\right)=\delta_{f=\square}\Big (\frac{X}{\zeta_{A}(2)} \Big )\prod_{\substack{P\mid
f}}\left(\frac{|P|}{|P|+1}\right)+O\left(X^{1/2}|f|^{1/4}\right).
\end{equation}
Here $\delta_{f=\square}=1$ if $f$ is a square and $\delta_{f=\square}=0$ otherwise.
\end{lemma}
\begin{proof}
  The formula in \eqref{meancharsum} follows from \cite[Lemma 4.1]{Andrade16}.  This it remains to prove \eqref{sumAd}.  Let $\mu(f)$ be the M\"obius function in $A$.  An application of Cauchy's residue theorem gives that, for $r$ small enough,
\begin{align}
\label{sumdint}
\begin{split}
\sum_{\substack{D\in\mathcal{H}_{2g+1,q} }}I(D)^{-k}\left(\frac{D}{f}\right)=&\sum_{\substack{\deg(D)=2g+1 }}\mu^2(D)I(D)^{-k}\left(\frac{D}{f}\right) \\
=&  \frac 1{2\pi i } \int\limits_{|u|=r} \sum_{\substack{D}}\mu^2(D)I(D)^{-k}\left(\frac{D}{f}\right) \frac {u^{\deg(D)} \dif u}{u^{2g+2}}.
\end{split}
\end{align}

Now, the multiplicativity of the summands gives
\begin{align*}
 \sum_{\substack{D}}\mu^2(D)I(D)^{-k}\left(\frac{D}{f}\right)u^{\deg(D)}=\prod_{P} \left( 1+I(P)^{-k}\leg {P}{f}u^{\deg(P)} \right)=\mathcal{L} \left( u, \leg {\cdot}{f} \right)Q_f(u),
\end{align*}
  where $\mathcal{L} \left( u, \leg {\cdot}{f} \right)$ is defined similar to that in \eqref{Ludef} and where
\begin{align*}
 Q_f(u)=\prod_{P} \left( 1-\leg {P}{f}u^{\deg(P)} \right) \left( 1+I(P)^{-k}\leg {P}{f}u^{\deg(P)} \right).
\end{align*}
  Note that $I(P)^{-k}=1+O(\frac {1}{|P|})$ so that 
  $$Q_f(u)=\prod_{P}\Big(1+O\Big( \frac {u^{\deg(P)}}{|P|} \Big)\Big).$$  
  The above then implies that $Q_f(u)$ is analytic in the region $|u|<1$. \newline

  We now shift the contour of integration in \eqref{sumdint} to $|u|=q^{-1/2-\varepsilon}$. We encounter no pole in this process when $f \neq \square$ and a simple pole at $u=1/q$ when
 $f=\square$. Note that it is immediate from \eqref{Zudef} that the residue of $\mathcal{Z}(u)$ at $u=1/q$ equals $-1/q$.  This allows us to deduce that the contribution of the residue due to the shifting of contour is
\begin{align*}
 X\prod_{P}(1-|P|^{-1})(1+I(P)^{-k}|P|^{-1}) \times \prod_{P|f}(1+I(P)^{-k}|P|^{-1})^{-1},
\end{align*}

Furthermore, it follows from the proof of \cite[Theorem 3.3]{AT14} (see also \cite[p. 6144]{Florea17}) that for a square-free $h \in A$ of degree $2d$ or $2d+1$, for any $|u|<q^{-1/2}$ and any $\varepsilon>0$,
\begin{align*}
 \mathcal{L}(u, \chi_h) \leq  e^{2d/\log_q(d)+4\sqrt{qd}} \ll_{\varepsilon} |h|^{\varepsilon}.
\end{align*}

We then write $(-1)^{\deg(f)}f=f_1f^2_2$ with $f_1$ square-free and apply the quadratic reciprocity law \eqref{quadrecp} to see that for any $|u|<1/2$ and any $\varepsilon>0$,
\begin{align*}
 \mathcal{L} \left( u, \leg {\cdot}{f} \right)=\mathcal{L} \left( u, \leg {(-1)^{\deg(f)}f}{\cdot} \right)=\mathcal{L}(u, \chi_{f_1})\prod_{P | f_2}(1-\chi_{f_1}(P)u^{\deg(P)}) \ll d_A(f_2)|f_1|^{\varepsilon} \ll_{\varepsilon} |f|^{\varepsilon}.
\end{align*}
Here $d_A(f)$ is the divisor function on $A$ and $d_A(f) \ll |f|^{\varepsilon}$ (see \cite[(3.3)]{G&Zhao12}).  Hence, the contribution from the integration on the new contour can be absorbed in the error term given in \eqref{sumAd}. This completes the proof of the lemma.
\end{proof}

\subsection{Perron’s formula}

The following analogue of Perron’s formula in function fields (see \cite[(2.6)]{Florea17-2}) is an easy consequence of Cauchy's residue theorem.
\begin{lemma}
\label{lem4}
 Suppose that the power series $\sum^{\infty}_{n=0} a(n)u^n$ is absolutely convergent in $|u| \leq  r < 1$.  Then for integers $N \geq 0$,
\begin{align}
\label{perron1}
\begin{split}
 \sum_{n \leq N}a(n)=& \frac 1{2\pi i}\oint\limits_{|u|=r} \Big ( \sum^{\infty}_{n=0}a(n)u^n\Big )\frac {\dif u}{(1-u)u^{N+1}}.
\end{split}
\end{align}
\end{lemma}

\section{Proof of Theorem \ref{t1}}
\label{sec 2'}

\subsection{Initial treatment} The proof follows closely that of \cite[Theorem 1.1]{G&Zhao2024-3}. We also use the ideas from \cites{Curran, Szab}. The key point is to apply the upper bounds method of Soundararajan-Harper as in \cites{Sound2009, Harper} to bound $\log |L(1/2, \chi_{D})|$ for $D \in \mathcal{H}_{2g+1, q}$ in terms of a short Dirichlet polynomial over the primes. Such a bound can be derived based the proofs of \cite[Theorem 3.3]{AT14}, \cite[Proposition 4.3]{BFK} and \cite[Lemma 3.1]{DFL}. As already pointed out in the Introduction, our proof of Theorem \ref{t1} differs from those in \cites{Curran, Szab} because there is a contribution from the prime squares. \newline
  
  Observe from \eqref{LchiDdef} that $L(s,\chi_{D})$  is a polynomial of degree $2g$ for $D \in H_{2g+1, q}$.  We may thus proceed as in the proof of \cite[Proposition 4.3]{BFK} by setting $m=2g+1$ there and also make use of the proof of \cite[Theorem 3.3]{AT14} to obtain the following result, analogous to \cite[Proposition 4.3]{BFK}.
\begin{proposition}
\label{prop-ub}
Let $D \in \mathcal{H}_{2g+1, q}$ and $m =2g+1$.  For $h \leq m$ and $z$ with $\Re(z) \geq 0$,  we have
\begin{align}
\label{logLupperbound}
\log \big| L(\tfrac{1}{2}+z, \chi_{D}) \big| \leq \frac{m}{h} + \frac{1}{h} \Re \bigg(  \sum_{\substack{j \geq 1 \\ \deg(P^j) \leq h}} \frac{
\chi_{D}(P^j) \log q^{h - j \deg(P)}}{|P|^{j ( 1/2+z+1/(h \log q) )} \log q^j} \bigg).
\end{align}
\end{proposition}

Lemma \ref{RS} implies that the contribution from terms on the right-hand side of \eqref{logLupperbound} with $j \geq 3$ is $O(1)$.  Thus, from \eqref{logLupperbound} by setting $X=q^{2g+1}, x=q^h$ and $\sigma+it=1/2+z$ with $\sigma, t \in \mr$, we see that, for $\sigma \geq 1/2$ and $x \leq X$,
\begin{align}
\label{logLbound}
\begin{split}
 \log  |L(\sigma+it, \chi)| \leq & \Re \left( \sum_{\substack{  |P| \leq x }} \frac{\chi_D (P)}{|P|^{\sigma+it+1/\log x}}
 \frac{\log (x/|P|)}{\log x} +
 \sum_{\substack{  |P| \leq x^{1/2} }} \frac{\chi_D (P^2)}{|P|^{2\sigma+2it+2/\log x}}  \frac{\log (x/|P|^2)}{2\log x} \right)
 +\frac{\log X}{\log x} + O(1) \\
 \leq  & \Re \left( \sum_{\substack{  |P| \leq x }} \frac{\chi_D (P)}{|P|^{\sigma+it+1/\log x}}
 \frac{\log (x/|P|)}{\log x} + \frac 12
 \sum_{\substack{  |P| \leq x^{1/2} }} \frac{\chi_D (P^2)}{|P|^{2\sigma+2it+2/\log x}} \right)
 +\frac{\log X}{\log x} + O(1),
\end{split}
 \end{align}
   where the last estimation above follows from \eqref{logp}. \newline

Again, \eqref{logp} renders that
\begin{align}
\label{sumprimesquaresimplfied}
\begin{split}
 \sum_{\substack{  |P| \leq x^{1/2} }} \Big ( \frac{\chi_D (P^2)}{|P|^{2\sigma+2it}}-\frac{\chi_D (P^2)}{|P|^{2\sigma+2it+2/\log x}}  \Big ) \ll \sum_{\substack{  |P| \leq x^{1/2} }}\frac {\log |P|}{|P|\log x} = O(1).
\end{split}
 \end{align}

Now \eqref{logLbound} and \eqref{sumprimesquaresimplfied} give
\begin{align}
\label{logLboundsimplified}
\begin{split}
 \log  |L(\sigma+it, \chi)| \leq & \Re \left( \sum_{\substack{  |P| \leq x }} \frac{\chi_D (P)}{|P|^{\sigma+it+1/\log x}}
 \frac{\log (x/|P|)}{\log x} + \frac 12
 \sum_{\substack{  |P| \leq x^{1/2} }} \frac{\chi_D (P^2)}{|P|^{2\sigma+2it}} \right)
 +\frac{\log X}{\log x} + O(1).
\end{split}
 \end{align}

Moreover,
\begin{align}
\label{sump}
\begin{split}
 \frac 12\sum_{\substack{  |P| \leq x^{1/2} }} \frac{\chi_D (P^2)}{|P|^{2\sigma+2it}} =& \frac 12\sum_{\substack{  |P| \leq x^{1/2} }} \frac{1}{|P|^{2\sigma+2it}}
 -\frac 12\sum_{\substack{|P|\leq x^{1/2} \\ P|D}} \frac{1}{|P|^{2\sigma+2it}} \\
 \leq & \frac 12\sum_{\substack{  |P| \leq x^{1/2} }} \frac{1}{|P|^{2\sigma+2it}}+\sum_{\substack{P|D}} \frac{1}{2|P|}.
\end{split}
\end{align}
  Using the inequality $x \leq -\log(1-x)$ for any $0<x<1$, the last expression above is
\begin{align}
\label{sump1}
\begin{split}
\leq \frac 12\sum_{\substack{  |P| \leq x^{1/2} }} \frac{1}{|P|^{2\sigma+2it}}-\sum_{\substack{P|D}}  \log \Big( 1- \frac{1}{2|P|} \Big).
\end{split}
\end{align}

   We deduce readily, from \eqref{logLboundsimplified}--\eqref{sump1}, the following analogous of \cite[Proposition 2]{Szab}, which estimates from above sums with $\log |L(1/2+it, \chi_D)|$ for various $t$.
\begin{proposition}
\label{prop1}
With the notation as above and assuming the truth of GRH, let $k\geq 1$ be a fixed integer and ${\bf a}=(a_1,\ldots ,a_{k}),\ t=(t_1,\ldots ,t_{k})$
 real $k$-tuples such that $a_i \geq 0$ for all $i$. For any monic polynomial $f$, let
$$H(f)=:\frac{1}{2}\Re \Big( \sum^{k}_{m=1}a_m|f|^{-it_m} \Big).$$
Set $a=a_1+\cdots+ a_{k}$.  Then, for large $X=q^{2g+1}$ and $\sigma \geq 1/2$,
\begin{align}
\label{basicest}
\begin{split}
 \sum^{2k}_{j=1}a_j \log| I(D)L(\sigma+it_j,\chi_D)|  \leq 2 \sum_{|P|\leq x} \frac{H(P)\chi_D(P)}{|P|^{\sigma+1/\log x}}\frac{\log (x/|P|)}{\log x}+\sum_{|P|\leq x^{1/2}} \frac{H(P^2)}{|P|^{2\sigma}}+a\frac{\log X}{\log x}+O(1).
\end{split}
\end{align}
\end{proposition}

   We also note the following upper bounds on moments of quadratic twists of $L$-functions, which is a consequence of \cite[Theorem 2.7]{Florea17-2}.
\begin{lemma}
\label{prop: upperbound}
With the notation as in Proposition~\ref{prop1}, we have for large $X=q^{2g+1}$,
\begin{align*}
   \sum_{D\in\mathcal{H}_{2g+1,q}}\big| L\big(\tfrac 12+it_1,\chi_D \big) \big|^{a_1} \cdots \big| L\big(\tfrac 12+it_{k},\chi_D  \big) \big|^{a_{k}} \ll_{{\bf a}} &  X(\log X)^{O(1)}.
\end{align*}
\end{lemma}

\subsection{Completion of the proof}

 Recall that we set $X=q^{2g+1}$. Following the ideas of A. J. Harper in \cite{Harper}, we define, for a large number $M$ depending on ${\bf a}$ only,
\begin{align*}
\begin{split}
 \alpha_{0} = 0, \;\;\;\;\; \alpha_{j} = \frac{20^{j-1}}{(\log\log X)^{2}}, \;\;\; \mbox{for all} \; j \geq 1, \quad
\mathcal{J} = \mathcal{J}_{{\bf a},X} = 1 + \max\{j : \alpha_{j} \leq 10^{-M} \} .
\end{split}
\end{align*}

   We set $P_j=(X^{\alpha_{j-1}}, X^{\alpha_{j}}]$ for $1 \leq j \leq \mathcal{J}$.  Lemma \ref{RS} gives that, for $X=q^{2g+1}$ large enough,
\begin{align*}
\begin{split}
 \sum_{P  \in P_{1}} \frac{1}{|P|} \leq & \log\log X=\alpha^{-1/2}_1,  \\
 \sum_{ P \in P_{j+1}} \frac{1}{|P|}
 =& \log \alpha_{j+1} - \log \alpha_{j} + o(1) =  \log 20 + o(1) \leq 10, \quad 1 \leq j \leq \mathcal{J}-1, \; \mbox{and} \\
\mathcal{J}-j \leq & \frac{\log(1/\alpha_{j})}{\log 20}.
\end{split}
\end{align*}

   We define $\lceil x \rceil$ to be $\min \{ n \in \intz : n \geq x\}$ for any real number $x$ and $\ell_j =2\lceil e^{B}\alpha^{-3/4}_j \rceil$ with $B$ being a large number to be chosen later depending on ${\bf a}$ only.  Note that $\ell_j$ are all even integers.  We set for any non-negative integer $\ell$ and any real number $x$,
\begin{equation*}
E_{\ell}(x) = \sum_{j=0}^{\ell} \frac{x^{j}}{j!}.
\end{equation*}

  We define three totally multiplicative functions $H(f, x),  H_1(f, x), s(f, x)$ whose values at primes $P$ are set to be
\begin{align*}
\begin{split}
  H(P, x)=& \frac{2 H(P)}{a|P|^{1/\log x}}, \quad H_1(P, x)= \frac{4H(P^2)}{a^2|P|^{2/\log x}}, \quad
  s(P,x)= \frac{\log (x/|P|)}{\log x}.
\end{split}
\end{align*}

First, we see that for any $f \in A$,
\begin{align*}
\begin{split}
  |H(f, x)| \leq 1.
\end{split}
\end{align*}

Let $w(n)$ denote the multiplicative function such that $w(P^{\alpha}) = \alpha!$ for prime powers $P^{\alpha}$. We set $x=X^{\alpha_j}$ and $z=1/2$ in \eqref{logLupperbound} to deduce that
\begin{align*}
\begin{split}
 & \sum^{k}_{m=1}a_m\log |I(D)L(\tfrac 12+it_m,\chi_D)| \le a\sum^{j}_{l=1} {\mathcal M}_{l,j}(D)+\frac {a^2}{4}\sum_{|P|\leq X^{\alpha_j/2}}
 \frac{H_1(P^2,X^{\alpha_j})}{|P|}+a\alpha^{-1}_j+O(1),
\end{split}
\end{align*}
  where
\[ {\mathcal M}_{l,j}(D) = \sum_{P\in P_l}\frac{H(P,X^{\alpha_j}) \chi_D(P)}{\sqrt{|P|}}s(P, X^{\alpha_j}), \quad 1\leq l \leq j \leq \mathcal{J}. \]

  We also define the following sets:
\begin{align*}
  \mathcal{S}(0) =& \{ D\in\mathcal{H}_{2g+1,q} : |a{\mathcal M}_{1,l}(D)| > \frac {\ell_{1}}{10^3} \; \text{ for some } 1 \leq l \leq \mathcal{J} \} ,   \\
 \mathcal{S}(j) =& \{ (D\in\mathcal{H}_{2g+1,q}  : |a{\mathcal M}_{m,l}(D)| \leq
 \frac {\ell_{m}}{10^3}  \; \mbox{for all} \; 1 \leq m \leq j, \; \mbox{and} \; m \leq l \leq \mathcal{J}, \\
 & \;\;\;\;\; \text{but }  |a{\mathcal M}_{j+1,l}(D)| > \frac {\ell_{j+1}}{10^3} \; \text{ for some } j+1 \leq l \leq \mathcal{J} \} ,  \quad  1\leq j \leq \mathcal{J}, \\
 \mathcal{S}(\mathcal{J}) =& \{D\in\mathcal{H}_{2g+1,q}  : |a{\mathcal M}_{m,
\mathcal{J}}(D)| \leq \frac {\ell_{m}}{10^3} \; \mbox{for all} \; 1 \leq m \leq \mathcal{J}\}.
\end{align*}

  We now proceed using the approach of \cite[Theorem 1.1]{G&Zhao2024-3}. As our proof follows by a straightforward modification of that of \cite[Theorem 1.1]{G&Zhao2024-3}, we shall omit some details here. Let $\Omega(f)$ denote the number of prime powers dividing $f \in A$ to see that
\begin{align}
\label{S0est}
\begin{split}
 \sum_{\substack{D \in \mathcal{S}(0)}}1   \leq &  \sum_{\substack{D\in\mathcal{H}_{2g+1,q} }}\sum^{\mathcal{J}}_{l=1}
\Big ( \frac {10^3}{\ell_1}{|a\mathcal
M}_{1, l}(D)| \Big)^{2\lceil 1/(10^3\alpha_{1})\rceil } \\
= & \sum^{\mathcal{J}}_{l=1}
 \Big ( \frac {a \cdot 10^3}{\ell_1} \Big)^{2\lceil 1/(10^3\alpha_{1})\rceil }  \sum_{ \substack{ f \\ \Omega(f) = 2\lceil 1/(10^3\alpha_{1})\rceil \\ P|f \implies
P \in P_1}}
\frac{(2\lceil 1/(10^3\alpha_{1})\rceil  )!s(f, X^{\alpha_l})}{\sqrt{f}}\frac{H(f,X^{\alpha_l})
  }{w(f)} \sum_{\substack{D\in\mathcal{H}_{2g+1,q} }}\chi_D(f).
\end{split}
\end{align}

   We apply Lemma \ref{prsum} to evaluate the inner-most sum on the right-hand side of \eqref{S0est}.  Thus, similar to \cite[(3.13)]{G&Zhao2024-3},
 \begin{align}
\label{S0upperbound}
\begin{split}
 \sum_{\substack{D \in \mathcal{S}(0)}}1   \ll  Xe^{-\alpha_1^{-1}/20}=Xe^{-(\log\log X)^{2}/20}  .
\end{split}
\end{align} 

  We now deduce via H\"older's inequality that
\begin{align}
\label{LS0bound}
\begin{split}
  \sum_{\substack{D \in \mathcal{S}(0)}} & \big| L\big(\tfrac12+it_1,\chi_D \big) \big|^{a_1} \cdots \big| L\big(\tfrac12+it_{k},\chi_D  \big) \big|^{a_{k}} 
\leq   \Big ( \sum_{\substack{D \in \mathcal{S}(0)}}1  \Big )^{1/2} \Big (
 \sum_{\substack{D \in \mathcal{S}(0)}}\big| L\big(\tfrac12+it_1,\chi_D \big) \big|^{2a_1} \cdots \big| L\big(\tfrac12+it_{k},\chi_D  \big) \big|^{2a_{k}} \Big)^{1/2}.
\end{split}
\end{align}

  Note that by \eqref{sumcosp}, we have for any $t \in \mr$, 
\begin{align*}
 \log X \gg \Big|\zeta_A \Big(1+it+\frac 1{\log X}\Big) \Big| \gg  \min \Big( \log X, \frac {1}{\overline {|t| \log q}} \Big) \gg 1.
\end{align*}  
  
  It then follows from Lemma \ref{prop: upperbound}, \eqref{S0upperbound} and \eqref{LS0bound} that
\begin{align*}
\sum_{\substack{D \in \mathcal{S}(0)}} & \big| L\big(\tfrac12+it_1,\chi_D \big) \big|^{a_1} \cdots \big| L\big(\tfrac12+it_{k},\chi_D  \big) \big|^{a_{k}} \\
   \ll & X(\log X)^{(a_1^2+\cdots +a_{k}^2)/4} \prod_{1\leq j<l \leq k} \Big|\zeta_A \Big(1+i(t_j-t_l)+\frac 1{\log X}\Big) \Big|^{a_ja_l/2}\Big|\zeta_A \Big(1+i(t_j+t_l)+\frac 1{\log X} \Big) \Big|^{a_ja_l/2} \\
& \hspace*{2cm} \times \prod_{1\leq j\leq k} \Big|\zeta_A \Big(1+2it_j+\frac 1{\log X}\Big) \Big|^{a^2_j/4+a_j/2}.
\end{align*}

  It thus remains to show that
\begin{align}
\label{sumovermj}
\begin{split}
\sum_{j=1}^{\mathcal{J}} \  \sum_{\substack{D \in \mathcal{S}(j)}} & \big| L\big(\tfrac12+it_1,\chi_D \big) \big|^{a_1} \cdots \big| L\big(\tfrac12+it_{k},\chi_D  \big) \big|^{a_{k}}\\
   \ll & X(\log X)^{(a_1^2+\cdots +a_{k}^2)/4} \prod_{1\leq j<l \leq k} \Big|\zeta_A \Big(1+i(t_j-t_l)+\frac 1{\log X}\Big) \Big|^{a_ja_l/2}\Big|\zeta_A \Big(1+i(t_j+t_l)+\frac 1{\log X} \Big) \Big|^{a_ja_l/2} \\
& \hspace*{2cm} \times \prod_{1\leq j\leq k} \Big|\zeta_A \Big(1+2it_j+\frac 1{\log X}\Big) \Big|^{a^2_j/4+a_j/2}.
\end{split}
\end{align}

 We deduce from \eqref{basicest} that for a fixed $j$ with $1 \leq j \leq \mathcal{J}$,   
\begin{align}
\label{zetaNbounds}
\begin{split}
 \big| L\big(\tfrac12+ it_1&, \chi_D \big)  \big|^{a_1} \cdots \big| L\big(\tfrac12+it_{k},\chi_D  \big) \big|^{a_{k}} \\
\ll & I(D)^{-a}\exp \left(\frac {(Q+1)a}{\alpha_j} \right)\exp \left(\sum_{|P|\leq X^{\alpha_j/2}} \frac{a^2H_1(P,X^{\alpha_j})}{4|P|} \right) \exp \Big (
 a\sum^j_{l=1}{\mathcal M}_{l,j}(D)\Big )\Phi \Big( \frac d{X} \Big).
\end{split}
\end{align}

 We apply the Taylor formula with integral remainder to see that
\begin{align*}
\begin{split}
\exp \Big ( a {\mathcal M}_{l,j}(D) \Big )  =E_{\ell_l}( a{\mathcal M}_{l,j}(D)) \left( 1+   O(e^{-\ell_l}) \right).
\end{split}
 \end{align*}

  Inserting the above into \eqref{zetaNbounds}, we obtain that
\begin{align*}
\begin{split}
 \big| L\big(\tfrac12+it_1 & ,\chi_D \big) \big|^{a_1} \cdots \big| L\big(\tfrac12+it_{k},\chi_D  \big) \big|^{a_{k}}  \\
\ll &  \exp \left(\frac {(Q+1)a}{\alpha_j} \right)\exp \left(\sum_{|P|\leq X^{\alpha_j/2}}\frac{a^2H_1(P,X^{\alpha_j})}{4|P|} \right) I(D)^{-a}\prod^{j}_{l=1}E_{\ell_l}( a{\mathcal M}_{l,j}(D)) .
\end{split}
\end{align*}

 It follows from the above and the description on $\mathcal{S}(j)$  that for a fixed $j \geq 1$,
\begin{align}
\label{upperboundprodE0}
\begin{split}
 & \sum_{\substack{D \in \mathcal{S}(j)}}  \big| L\big(\tfrac12+it_1,\chi_D \big) \big|^{a_1} \cdots \big| L\big(\tfrac12+it_{k},\chi_D  \big) \big|^{a_{k}} 
\ll \exp \left(\frac {(Q+1)a}{\alpha_j} \right)\exp \left(\sum_{|P|\leq X^{\alpha_j/2}} \frac{a^2H_1(P,X^{\alpha_j})}{4|P|} \right) \sum^{ \mathcal{J}}_{v=j+1}S_v,
\end{split}
\end{align}
  where
\begin{align*}
\begin{split}
S_v =: \sum_{\substack{D\in\mathcal{H}_{2g+1,q} }}I(D)^{-a}\prod^{j}_{l=1}E_{\ell_l}(a{\mathcal M}_{l,j}(D))  \Big ( \frac {10^3}{\ell_{j+1}}|a {\mathcal M}_{j+1,v}(D)|\Big)^{2\lceil 1/(10\alpha_{j+1})\rceil }.
\end{split}
\end{align*}

  We now focus on the evaluation of $S_v$ for a fixed $v$ by expanding the factors in $S_v$ into Dirichlet series. Similar to \cite[(3.24)]{G&Zhao2024-3}, we may write for simplicity that
\begin{align}
\label{Suexpression}
 S_v= \sum_{|f|  \leq X^{1/10}} c_f  \sum_{\substack{D\in\mathcal{H}_{2g+1,q} }} I(D)^{-a}\chi_D(f),
\end{align}
  where $|c_{f}| \ll  X^{\varepsilon}$. \newline

  We now apply Lemma \ref{prsum} to evaluate the inner-most sum in \eqref{Suexpression} by noting that the contribution from the $O$-term in \eqref{sumAd} is
negligible, while the contribution from the main term in \eqref{sumAd} to $S_v$ is
\begin{align*}
\begin{split}
  \ll X \prod^j_{l=1} & \Big (\sum_{\substack{f_l=\square}} \frac{H(f_l,X^{\alpha_j})s(f_l,X^{\alpha_j})}{\sqrt{f_l}} \frac{a^{\Omega(f_l)}}{w(f_l)}  b_l(f_l) \prod_{P|f_l}(1+I(P)^{-a}|P|^{-1})^{-1}\Big )  \Big ( \Big ( \frac {a \cdot 10^3}{\ell_{j+1}} \Big)^{2\lceil 1/(10^3\alpha_{j+1})\rceil }  \\
& \times \sum_{ \substack{ f_u=\square \\ \Omega(f_u) = 2\lceil 1/(10^3\alpha_{j+1})\rceil \\ P|f_u \implies
P \in P_{j+1}}}
\frac{(2\lceil 1/(10^3\alpha_{j+1})\rceil  )!s(f_u, X^{\alpha_u})}{\sqrt{f_u}}\frac{H(f_u, X^{\alpha_u})
  }{w(f_u)}\prod_{P|f_u}(1+I(P)^{-a}|P|^{-1})^{-1} \Big ),
\end{split}
\end{align*}
  where $b_j(f)$, $1 \leq j \leq \mathcal{J}$ are functions such that $b_j(f)=0$ or $1$ and $b_j(f)=1$ if and only if $\Omega(f) \leq \ell_j$ and the primes dividing $f$ are all from the interval $P_j$. \newline

   Using the observation that $\prod_{P|f_l}(1+I(P)^{-k}|P|^{-1})^{-1} \leq 1$ and $H(f, X^{\alpha_j}), H(f,  X^{\alpha_u}) \geq 0$ when $f=\square$, we see that the expression above is
\begin{align*}
\begin{split}
  \ll X \prod^j_{l=1} & \Big (\sum_{\substack{f_l=\square}} \frac{H(f_l,X^{\alpha_j})s(f_l,X^{\alpha_j})}{\sqrt{f_l}} \frac{a^{\Omega(f_l)}}{w(f_l)}  b_l(f_l) \Big ) \Big ( \Big ( \frac {a \cdot 10^3}{\ell_{j+1}} \Big)^{2\lceil 1 /(10^3\alpha_{j+1})\rceil } \\
& \times \sum_{ \substack{ f_u=\square \\ \Omega(f_u) = 2\lceil 1/(10^3\alpha_{j+1})\rceil \\ P|f_u \implies
P \in P_{j+1}}}
\frac{(2\lceil 1/(10^3\alpha_{j+1})\rceil  )!s(f_u, X^{\alpha_u})}{\sqrt{f_u}}\frac{H(f_u, X^{\alpha_u})
  }{w(f_u)}\Big ).
\end{split}
\end{align*}

  We evaluate
\begin{align}
\label{sumsqurei}
   \sum_{\substack{f_l=\square}} \frac{H(f_l, X^{\alpha_j})s(f_l,X^{\alpha_j})}{\sqrt{f_l}} \frac{a^{\Omega(f_l)}}{w(f_l)}  b_l(f_l)
\end{align}
  by noting that the factor $b_l(f_l)$ limits $f_l$ to have all prime factors in $P_l$ such that $\Omega(f_l) \leq \ell_l$. If we remove the restriction on $\Omega(f_l)$, then \eqref{sumsqurei} becomes
\begin{align}
\label{sumsqureififree}
 \sum_{\substack{f_l=\square \\ P | f_l \Rightarrow P \in P_l}} \frac{H(f_l, X^{\alpha_j})s(f_l,X^{\alpha_j})}{\sqrt{f_l}} \frac{a^{\Omega(f_l)}}{w(f_l)}.
\end{align}
     On the other hand, using Rankin's trick by noticing that $2^{n_l-\ell_l}\ge 1$ if $\Omega(n_l) > \ell_l$, we see that the error incurred in this relaxation does not exceed
\begin{align}
\label{sumsqureierror1}
 \sum_{\substack{f_l=\square \\ P | f_l \Rightarrow P \in P_l}} \frac{2^{\Omega(f_l)-\ell_l}|H(f_l,X^{\alpha_j})s(f_l,X^{\alpha_j})|}{\sqrt{f_l}} \frac{a^{\Omega(f_l)}}{w(f_l)}.
\end{align}
  
   As the expressions in \eqref{sumsqureififree} and \eqref{sumsqureierror1} are now multiplicative functions of $f_l$, we may evaluate them in terms of products over primes.  From Lemma \ref{RS}, the observation that for $|P| \leq X^{\alpha_j}$, 
\begin{align*}
  s(P, X^{\alpha_j})=1+O\Big( \frac {\log |P|}{\log X^{\alpha_j}} \Big ),
\end{align*}
  we deduce that
\begin{align*}
\begin{split}
 & \sum_{\substack{f_l=\square}} \frac{H(f_l,X^{\alpha_j})s(f_l,X^{\alpha_j})}{\sqrt{f_l}} \frac{a^{\Omega(f_l)}}{w(f_l)}  b_l(f_l) 
 \leq 
  \Big (1+O(2^{-\ell_l/2})\Big )\exp \Big ( \sum_{P \in  P_l } \frac {a^2H^2(P,X^{\alpha_j})}{2|P|}+O\Big( \sum_{P \in  P_l } \Big (\frac {\log |P|}{|P|\log X^{\alpha_j}}+\frac 1{|P|^{2}}  \Big) \Big)\Big ).
\end{split}
\end{align*}

   By similar arguments, we see that upon taking $B$ large enough,
\begin{align*}
\begin{split}
   \Big ( \frac {a \cdot 10^3}{\ell_{j+1}} \Big)^{2\lceil 1/(10^3\alpha_{j+1})\rceil } \sum_{ \substack{ f_u=\square \\ \Omega(f_u) = 2\lceil 1/(10^3\alpha_{j+1})\rceil \\ P|f_u \implies
P \in P_{j+1}}}
\frac{(2\lceil 1/(10^3\alpha_{j+1})\rceil  )!s(f_u, X^{\alpha_u})}{\sqrt{f_u}}\frac{H(f_u,X^{\alpha_u})
  }{w(f_u)} \ll & e^{-10^3(Q+1)a/(2\alpha_{j+1})}.
\end{split}
\end{align*}   
   
   It follows from the above and Lemma \ref{RS}, 
\begin{align*}
\begin{split}
  S_v
\ll & e^{-10^3(Q+1)a/(2\alpha_{j+1})}X \exp \Big (\sum_{P \in  \bigcup^j_{l=1}P_l }\frac {a^2H^2(P, X^{\alpha_j})}{2|P|}
 +O\Big(\sum_{|P| \in  \bigcup^j_{l=1}P_l }\Big (\frac {\log |P|}{|P|\log X^{\alpha_j}}+\frac 1{p^{2}}\Big ) \Big )\Big ) \\
 \ll& e^{-10^3(Q+1)a/(2\alpha_{j+1})}X \exp \Big (\sum_{P \in  \bigcup^j_{l=1}P_l }\frac {a^2H^2(P, X^{\alpha_j})}{2p} \Big ).
\end{split}
\end{align*}
   
  We insert the above estimation into \eqref{upperboundprodE0} and make use of the observation that $20/\alpha_{j+1}=1/\alpha_j$.  Consequently
\begin{align}
\label{produpperboundoverSj}
\begin{split}
\sum_{\substack{D \in \mathcal{S}(j)}} & \big| L\big(\tfrac12+it_1,\chi_D \big) \big|^{a_1} \cdots \big| L\big(\tfrac12+it_{k},\chi_D  \big) \big|^{a_{k}}  \\
\ll & \exp \left(-\frac {(Q+1)a}{\alpha_j} \right)X
\exp \Big (\sum_{P \in  \bigcup^j_{l=1}P_l }\frac {a^2H^2(P,X^{\alpha_j})}{2p}+\sum_{|P|\leq X^{\alpha_j/2}}
\frac{a^2H_1(P,X^{\alpha_j})}{4|P|} \Big ) \\
\ll & X \exp \Big (-\frac {(Q+1)a}{\alpha_j} + \sum_{|P| \leq X}\frac {(2H(P))^2)}{2|P|}+\sum_{|P|\leq X} \frac{H(P^2)}{|P|} \Big ).
\end{split}
\end{align}
 
  We now apply Lemma \ref{mertenstype} to evaluate the last expression in \eqref{produpperboundoverSj} and note that summing over $j$ leads to a convergent series. 
  This enables us to derive the desired estimation in \eqref{sumovermj} and therefore completes the proof of Theorem \ref{t1}.

\section{Proof of Theorem \ref{quadraticmean}}
\label{sec 3}

The proof of Theorem \ref{quadraticmean} uses ideas in the proof of \cite[Theorem 3]{Szab}. Without loss of generality, we may assume that $\log_q Y$ is a positive integer.  Recall the definition of $\mathcal{L}(u,\chi_D)$ in \eqref{Ludef}.  We apply Perron’s formula in \eqref{perron1} to see that for a small $r>0$,
\begin{align}
\label{sumchi}
\sum_{|f| \leq  Y} \chi_D(f)=& \frac 1{2\pi i}\oint\limits_{|u|=r} \Big ( \sum_{f} \chi_D(f) u^{d(f)}\Big )\frac {\dif u}{(1-u)u^{\log_q Y+1}}=\frac 1{2\pi i}\oint\limits_{|u|=r} \frac {\mathcal{L}(u,\chi_D) \dif u}{(1-u)u^{\log_q Y+1}}.
\end{align}

  We shift the line of integration in \eqref{sumchi} to $|u|= q^{-1/2}$ without encountering any pole. Thus
\begin{align*}
   \sum_{|f| \leq  Y} \chi_D(f)=\frac 1{2\pi i}\oint\limits_{|u|=q^{-1/2}} \frac {\mathcal{L}(u,\chi_D) \dif u}{(1-u)u^{\log_q Y+1}}= \frac {Y^{1/2}}{2\pi i}\oint\limits_{|u|=1} \frac {\mathcal{L}(u/\sqrt{q},\chi_D) \dif u}{(1-u/\sqrt{q})u^{\log_q Y+1}}.
\end{align*}
    We deduce from the above that
\begin{align*}
  \sum_{D \in \mathcal{H}_{2g+1,q}} \Big |\sum_{|f| \leq  Y} \chi_D(f)\Big |^{2m} \ll & Y^m\sum_{D \in \mathcal{H}_{2g+1,q}} \Big |\oint\limits_{|u|=1} \frac {\mathcal{L}(u/\sqrt{q},\chi_D) \dif u}{(1-u/\sqrt{q})u^{\log_q Y+1}}\Big |^{2m} \ll Y^m\sum_{D \in \mathcal{H}_{2g+1,q}} \Big (\int\limits_{0}^{2\pi} \Big |\mathcal{L}(e^{it}/\sqrt{q},\chi_D) \Big |\dif t\Big )^{2m}.
\end{align*}

  It follows that to prove Theorem \ref{quadraticmean}, it suffices to establish the following.
\begin{proposition}
\label{t3prop}
 With the notation as above, we have for any real number $m \geq 3/2$,
\begin{align}
\label{finiteintest}
\begin{split}
 & \sum_{D \in \mathcal{H}_{2g+1,q}} \Big ( \Big |\int\limits_{0}^{2\pi} \Big |\mathcal{L}(e^{it}/\sqrt{q},\chi) \Big | \dif t \Big )^{2m} \ll  X(\log X)^{2m^2-m+1}.
\end{split}
\end{align}
\end{proposition}
\begin{proof}
As $|\mathcal{L}(e^{it}/\sqrt{q},\chi_D)\Big |=|\overline{\mathcal{L}(e^{i(2\pi-t)}/\sqrt{q}, \chi_D)}\Big |$, it follows that
\begin{align}
\label{finiteintest0}
\begin{split}
  \sum_{D \in \mathcal{H}_{2g+1,q}} \Big (\int\limits_{0}^{2\pi} \Big |\mathcal{L}(e^{it}/\sqrt{q},\chi_D)  \Big |\dif t\Big )^{2m}
 \ll & \sum_{D \in \mathcal{H}_{2g+1,q}}\Big ( \int\limits_{0}^{\pi} \Big |\mathcal{L}(e^{it}/\sqrt{q},\chi_D)  \Big | \dif t \Big )^{2m}.
\end{split}
\end{align}

  Moreover, the right-hand side of \eqref{finiteintest0} is
\begin{align}
\label{finiteintest1}
\begin{split}
\ll & \sum_{D \in \mathcal{H}_{2g+1,q}} \Big ( \int\limits_{0}^{\pi/2} \Big |\mathcal{L}(e^{it}/\sqrt{q},\chi_D)  \Big | \dif t\Big )^{2m}+\sum_{D \in \mathcal{H}_{2g+1,q}} \Big ( \int\limits_{\pi/2}^{\pi} \Big |\mathcal{L}(e^{it}/\sqrt{q},\chi_D)  \Big | \dif t\Big )^{2m} \\
 =& \sum_{D \in \mathcal{H}_{2g+1,q}} \Big ( \int\limits_{0}^{\pi/2} \Big |\mathcal{L}(e^{it}/\sqrt{q},\chi_D)  \Big | \dif t\Big )^{2m}+\sum_{D \in \mathcal{H}_{2g+1,q}} \Big ( \int\limits_{\pi/2}^{0} \Big |\mathcal{L}(e^{i(\pi-t)}/\sqrt{q},\chi_D)  \Big | \dif (\pi-t)\Big )^{2m} \\
 =& \sum_{D \in \mathcal{H}_{2g+1,q}} \Big ( \int\limits_{0}^{\pi/2} \Big |\mathcal{L}(e^{it}/\sqrt{q},\chi_D)  \Big | \dif t\Big )^{2m}+\sum_{D \in \mathcal{H}_{2g+1,q}} \Big ( \int\limits_{0}^{\pi/2} \Big |\mathcal{L}(-e^{-it}/\sqrt{q},\chi_D) \Big | \dif t \Big )^{2m}.
\end{split}
\end{align}

   Let $k \geq 1$ be a fixed integer. We treat the first sum of the last expression above by deducing via symmetry that for $2m \geq k+1$,
\begin{align}
\label{Lintdecomp}
    \Big ( \int\limits_{0}^{\pi} \Big |\mathcal{L}(e^{it}/\sqrt{q},\chi_D) \Big | \dif t \Big )^{2m}
      \ll \int\limits_{[0,\pi/2]^k}\prod_{a=1}^k|\mathcal{L}(e^{it_a}/\sqrt{q},\chi_D)| \cdot \bigg(\int\limits_{\mathcal{D} }|\mathcal{L}(e^{iv}/\sqrt{q},\chi_D)| \dif v \bigg)^{2m-k} \dif \mathbf{t},
\end{align}
where $\mathcal{D}=\mathcal{D}(t_1,\ldots,t_k)=\{ v\in [0,\pi/2]:|t_1-v|\leq |t_2-v|\leq \ldots \leq |t_k-v| \}$. \newline

  We let $\mathcal{B}_1=\big[-\frac{\pi}{2\log X},\frac{\pi}{2\log X}\big]$ and $\mathcal{B}_j=\big[-\frac{e^{j-1}\pi}{2\log X}, -\frac{e^{j-2}\pi}{2\log X}\big]
  \cup \big[\frac{e^{j-2}\pi}{2\log X}, \frac{e^{j-1}\pi}{2\log X}\big]$ for $2\leq j< \lfloor \log \log X\rfloor =:K$. We further denote
  $\mathcal{B}_K=[-\pi/2,\pi/2]\setminus \bigcup_{1\leq j<K} \mathcal{B}_j$. \newline

Observe that for any $t_1\in [0, \pi/2]$,  we have $\mathcal{D}\subset [0,\pi/2] \subset t_1+[-\pi/2,\pi/2]\subset \bigcup_{1\leq j\leq K} (t_1+\mathcal{B}_j)$. Thus
if we set $\mathcal{A}_j=\mathcal{B}_j\cap (-t_1+\mathcal{D})$, then $(t_1+\mathcal{A}_j)_{1\leq j\leq K}$ form a partition of $\mathcal{D}$.  From applying Hölder's inequality twice, emerges the bounds,
\begin{align}
\label{LintoverD}
\begin{split}
    & \bigg(\int\limits_{\mathcal{D}}|\mathcal{L}(e^{iv}/\sqrt{q},\chi_D)| \dif v\bigg)^{2m-k} \leq \bigg( \sum_{1\leq j\leq K} \frac{1}{j}\cdot  j \int\limits_{t_1+\mathcal{A}_j} |\mathcal{L}(e^{iv}/\sqrt{q},\chi_D)| \dif v \bigg)^{2m-k} \\
     \leq & \bigg(\sum_{1\leq j\leq K} j^{2m-k} \bigg( \int\limits_{t_1+\mathcal{A}_j} \big|\mathcal{L}(e^{iv}/\sqrt{q},\chi_D)\big| \dif v  \bigg)^{2m-k}\bigg)
     \bigg(\sum_{1\leq j\leq K } j^{-(2m-k)/(2m-k-1)} \bigg)^{2m-k-1} \\
     \ll & \sum_{1\leq j\leq K} j^{2m-k} \bigg( \int\limits_{t_1+\mathcal{A}_j} \big|\mathcal{L}(e^{iv}/\sqrt{q},\chi_D)\big| \dif v \bigg)^{2m-k} \\
     \leq & \sum_{1\leq j\leq K} j^{2m-k} |\mathcal{B}_j|^{2m-k-1} \int\limits_{t_1+\mathcal{A}_j} \big|\mathcal{L}(e^{iv}/\sqrt{q},\chi_D)\big|^{2m-k} \dif v.
\end{split}
\end{align}
For $\mathbf{t}=(t_1,\ldots,t_k)$, we write
$$\mathcal{L}(\mathbf{t},v)=\sum_{D \in \mathcal{H}_{2g+1,q}} \prod_{a=1}^k|\mathcal{L}(e^{it_a}/\sqrt{q},\chi_D)| \cdot |\mathcal{L}(e^{iv}/\sqrt{q},\chi_D)|^{2m-k}.$$
  We then deduce from \eqref{Lintdecomp} and \eqref{LintoverD} that
\begin{align*}
\begin{split}
  \sum_{D \in \mathcal{H}_{2g+1,q}} \Big ( \int\limits_{0}^{\pi} \Big |\mathcal{L}(e^{it}/\sqrt{q},\chi_D)  \Big | \dif t\Big )^{2m} \ll &
    \sum_{1\leq l_0\leq K} l_0^{2m-k} |\mathcal{B}_{l_0}|^{2m-k-1} \int\limits_{[0,\pi/2]^k}\int\limits_{t_1+\mathcal{A}_{l_0}}\mathcal{L}(\mathbf{t},v) \dif v \dif \mathbf{t}  \\
     \ll &  \sum_{1\leq l_0, l_1, \ldots l_{k-1}\leq K} l_0^{2m-k} |\mathcal{B}_{l_0}|^{2m-k-1} \int\limits_{\mathcal{C}_{l_0,l_1, \cdots, l_{k-1}}} \mathcal{L}(\mathbf{t},v) \dif v \dif \mathbf{t},
\end{split}
\end{align*}
where
$$\mathcal{C}_{l_0,l_1, \cdots, l_{k-1}}=\{(t_1,\ldots,t_k,v)\in [0,\pi/2]^{k+1}: v\in t_1+ \mathcal{A}_{l_0},\, |t_{i+1}-v|-|t_i-v|\in \mathcal{B}_{l_i},1 \ 1 \leq i \leq k-1\}.$$

The volume of the region $\mathcal{C}_{l_0,l_1, \cdots, l_{k-1}}$ is $\ll e^{l_0+l_1+\cdots+l_{k-1}}/(\log X)^k$. Also,
by the definition of $\mathcal{C}_{l_0,l_1, \cdots, l_{k-1}}$ and the observation that $t_i$, $v \geq 0$, $1\leq i \leq k$, we have $e^{l_0}/\log X \ll |t_1-v|\ll |t_1+v|\ll 1$ so that $w(|t_1\pm v|)\ll \log X/e^{l_0}$, where we define, for simplicity, $w(t)=\min (\log X, 1/\overline{|t|})$. We deduce from the definition of $\mathcal{A}_j$ that $|t_2-v|\geq |t_1-v|$, so that $1 \gg |t_2+v| \gg |t_2-v|= |t_1-v|+(|t_2-v|-|t_1-v|)\gg  e^{l_0}/\log X + e^{l_1}/\log X$, which implies that $w(|t_2\pm v|)\ll \log X/e^{\max(l_0,l_1) }$.
Similarly, $w(|t_i\pm v|)\ll \log X /e^{\max(l_0,l_1,\ldots, l_{i-1}) }$ for any $1 \leq i \leq k$.
Moreover, we have $\sum^{j-1}_{s=i}(|t_{s+1}-v|-|t_s-v|) \leq |t_j-t_i| \leq |t_j+t_i|$ for any $1 \leq i < j \leq k$, so that we have $w(|t_{j}\pm t_i|)\ll \log X /e^{\max(l_i,\ldots, l_{j-1} ) }$.
We also bound $w(|2t_i|), 1\leq i \leq k$ and $w(|2u|)$ trivially by $\log X$. \newline

  Observe that, upon setting $u=q^{-1/2}e^{i\theta}$ in the function $\mathcal{L}(u,\chi_D)$ defined in \eqref{Ludef},  we have that
$$ \mathcal{L} \Big( \frac {e^{i\theta}}{\sqrt{q}},\chi_D \Big) = L\Big(\frac 12- \frac {\theta i}{\log q}, \chi_D \Big).$$
 Using this in Theorem \ref{t1}, we readily deduce a restatement of the theorem in terms of $\mathcal{L}(\frac {u}{\sqrt{q}},\chi_D)$ that asserts under the same notation as Theorem \ref{t1},
\begin{align}
\label{mathcalLestimation}
\begin{split}
\sum_{D\in\mathcal{H}_{2g+1,q}} & \Big |\mathcal{L} \Big( \frac {e^{i\theta_1}}{\sqrt{q}},\chi_D \Big) \Big|^{a_1} \cdots \Big|\mathcal{L}\Big(\frac {e^{i\theta_{2k}}}{\sqrt{q}},\chi_D \Big) \Big|^{a_{k}} \\
 \ll & X(\log X)^{(a_1^2+\cdots +a_{k}^2)/4}  \prod_{1\leq j<l \leq k} \min \Big( \log X, \frac{1}{\overline{|\theta_j-\theta_l|}} \Big)^{a_ja_l/2}\min \Big( \log X, \frac {1}{\overline{ |\theta_j+\theta_l|}} \Big)^{a_ja_l/2} \\
 & \hspace*{2cm} \times \prod_{1\leq j\leq k} \min \Big( \log X, \frac {1}{\overline{|2\theta_j|}} \Big)^{a^2_j/4+a_j/2}.
\end{split}
\end{align}
 Here the implied constant depends on $k$ and $a_j$, but not on $X$ or the $t_j$'s. \newline

  It from \eqref{mathcalLestimation} that for $(t_1,\ldots,t_k,v)\in \mathcal{C}_{l_0,l_1, \cdots, l_{k-1}}$,
\begin{align*}
   L(\mathbf{t},v) \ll & X(\log X)^{((2m-k)^2+k)/4+(2m-k)^2/4+(2m-k)/2+3k/4}
     \bigg(\prod^{k-1}_{i=0}\frac{\log X}{e^{ \max(l_0,l_1,\ldots, l_{i}) }} \bigg)^{2m-k}
     \bigg(\prod^{k-1}_{i=1} \prod^{k}_{j=i+1}\frac{\log X}{e^{\max(l_i,\ldots, l_{j-1} ) } } \bigg) \\
     = & X(\log X)^{m(2m+1)}\exp\Big( -(2m-k)\sum^{k-1}_{i=0}\max(l_0,l_1,\ldots, l_{i})-\sum^{k-1}_{i=1} \sum^{k}_{j=i+1}\max(l_i,\ldots, l_{j-1} )\Big).
\end{align*}

 Moreover,  we have $ |\mathcal{B}_{l_0}|\ll e^{l_0}/\log X$, so that, for $k \geq 2$ and $2m \geq k+1$,
\begin{align}
\label{firstcase}
\begin{split}
      &  \sum_{D \in \mathcal{H}_{2g+1,q}}\Big ( \int\limits_{0}^{\pi} \Big |\mathcal{L}(e^{it}/\sqrt{q},\chi)  \Big | \dif t\Big )^{2m}
\ll \sum_{\substack{1\leq l_0 \leq K \\ 1\leq l_1, \ldots l_{k-1}\leq K}}  l_0^{2m-k} |\mathcal{B}_{l_0}|^{2m-k-1} \int\limits_{\mathcal{C}_{l_0,l_1, \cdots, l_{k-1}}} \mathcal{L}(\mathbf{t},u) \dif u \dif \mathbf{t} \\
    \ll & X(\log X)^{2m^2-m+1} \\
    &  \times \sum_{\substack{1\leq l_0 \leq K \\ 1\leq l_1, \ldots l_{k-1}\leq K}}  l_0^{2m-k}\exp\Big( (2m-k-1)l_0+\sum^{k-1}_{i=0}l_i-(2m-k)\sum^{k-1}_{i=0}\max(l_0,l_1,\ldots, l_{i})-\sum^{k-1}_{i=1} \sum^{k}_{j=i+1}\max(l_i,\ldots, l_{j-1} )\Big) \\
    = & X(\log X)^{2m^2-m+1}  \\
    &  \times \sum_{\substack{1\leq l_0 \leq K \\ 1\leq l_1, \ldots l_{k-1}\leq K}}  l_0^{2m-k}\exp\Big( -(2m-k)\sum^{k-1}_{i=1}\max(l_0,l_1,\ldots, l_{i})-\sum^{k-1}_{i=1} \sum^{k}_{j=i+2}\max(l_i,\ldots, l_{j-1} )\Big) \\
    \ll &   X(\log X)^{2m^2-m+1} \sum_{\substack{1\leq l_0 \leq K \\ 1\leq l_1, \ldots l_{k-1}\leq K}}  l_0^{2m-k}\exp\Big( -(2m-k)\sum^{k-1}_{i=1}\frac {\sum^{i}_{s=0}l_s}{i+1}-\sum^{k-1}_{i=1} \sum^{k}_{j=i+2}\frac {\sum^{j-1}_{s=i}l_s}{j-i}\Big)  \\
    \ll &   X(\log X)^{2m^2-m+1}.
\end{split}
\end{align}
We now set $k=2$ to see that the above estimation holds for $m\geq 3/2$. \newline

Note that the estimation given in \eqref{mathcalLestimation}  is still valid with $\theta_i$ replaced by $\pi-\theta_i$ on the left-hand side of \eqref{mathcalLestimation} while keeping
  $\theta_j, \theta_l$ intact on the right-hand side of \eqref{mathcalLestimation}. Using this, one checks that our arguments above can be used to show that
\begin{align}
\label{Intbound1}
\begin{split}
    &  \sum_{D \in \mathcal{H}_{2g+1,q}}\Big ( \int\limits_{0}^{\pi} \Big |\mathcal{L}(-e^{it}/\sqrt{q},\chi_D) \Big | \dif t \Big )^{2m} \ll X(\log X)^{2m^2-m+1}.
\end{split}
\end{align}
   We then deduce from \eqref{finiteintest0}, \eqref{finiteintest1}, \eqref{firstcase} and \eqref{Intbound1} that \eqref{finiteintest} holds. This completes the proof of the proposition.
\end{proof}

\vspace*{.5cm}

\noindent{\bf Acknowledgments.}  P. G. is supported in part by NSFC grant 12471003 and L. Z. by the FRG Grant PS71536 at the University of New South Wales.  The authors would like to thank the anonymous referee for his/her meticulous inspection of the paper and many valuable suggestions.

\bibliography{biblio}

\begin{bibdiv}
\begin{biblist}

\bib{AT14}{article}{
      author={Altu\u{g}, S.~A.},
      author={Tsimerman, J.},
       title={Metaplectic {R}amanujan conjecture over function fields with
  applications to quadratic forms},
        date={2014},
     journal={Int. Math. Res. Not. IMRN},
      number={13},
       pages={3465\ndash 3558},
}

\bib{Andrade16}{article}{
      author={Andrade, J.},
       title={Rudnick and {S}oundararajan's theorem for function fields},
        date={2016},
     journal={Finite Fields Appl.},
      volume={37},
       pages={311\ndash 327},
}

\bib{Andrade&Keating12}{article}{
      author={Andrade, J.~C.},
      author={Keating, J.~P.},
       title={The mean value of {$L(\frac12,\chi)$} in the hyperelliptic
  ensemble},
        date={2012},
     journal={J. Number Theory},
      volume={132},
      number={12},
       pages={2793\ndash 2816},
}

\bib{B&Gao2024}{article}{
      author={Baier, S.},
      author={Gao, P.},
       title={Upper bounds for shifted moments of {D}irichlet {$L$}-functions
  to a fixed modulus over function fields},
        date={Preprint},
        note={arXiv:2406.19606},
}

\bib{BerDiaPetWes}{article}{
      author={Bergström, J.},
      author={Diaconu, A.},
      author={Petersen, D.},
      author={Westerland, C.},
       title={Hyperelliptic curves, the scanning map, and moments of families
  of quadratic {$L$}-functions},
        date={Preprint},
        note={arXiv:2302.07664},
}

\bib{BFK}{article}{
      author={Bui, H.~M.},
      author={Florea, A.},
      author={Keating, J.~P.},
      author={Roditty-Gershon, E.},
       title={Moments of quadratic twists of elliptic curve {$L$}-functions
  over function fields},
        date={2020},
     journal={Algebra Number Theory},
      volume={14},
      number={7},
       pages={1853\ndash 1893},
}

\bib{Chandee}{article}{
      author={Chandee, V.},
       title={On the correlation of shifted values of the {R}iemann zeta
  function},
        date={2011},
     journal={Q. J. Math.},
      volume={62},
      number={3},
       pages={545\ndash 572},
}

\bib{CFKRS}{article}{
      author={Conrey, J.~B.},
      author={Farmer, D.~W.},
      author={Keating, J.~P.},
      author={Rubinstein, M.~O.},
      author={Snaith, N.~C.},
       title={Integral moments of {$L$}-functions},
        date={2005},
     journal={Proc. London Math. Soc. (3)},
      volume={91},
      number={1},
       pages={33\ndash 104},
}

\bib{Curran}{article}{
      author={Curran, M.~J.},
       title={Correlations of the {R}iemann zeta function},
        date={2024},
     journal={Mathematika},
      volume={70},
      number={4},
       pages={Paper No. e12268, 14 pp},
}

\bib{DarMai}{article}{
      author={Darbar, P.},
      author={Maiti, G.},
       title={Correlation of shifted values of {$L$}-functions in the
  hyperelliptic ensemble},
        date={2021},
     journal={Finite Fieds Appl.},
      volume={76},
       pages={Paper No. 101928, 38 pp.},
}

\bib{DFL}{article}{
      author={David, C.},
      author={Florea, A.},
      author={Lalin, M.},
       title={Nonvanishing for cubic {$L$}-functions},
        date={2021},
     journal={Forum Math. Sigma},
      volume={9},
       pages={Paper No. e69, 58pp},
}

\bib{Diaconu19}{article}{
      author={Diaconu, A.},
       title={On the third moment of {$L(\frac{1}{2},\chi_d)$} {I}: {T}he
  rational function field case},
        date={2019},
        ISSN={0022-314X},
     journal={J. Number Theory},
      volume={198},
       pages={1\ndash 42},
}

\bib{Florea17-2}{article}{
      author={Florea, A.},
       title={The fourth moment of quadratic {D}irichlet {$L$}-functions over
  function fields},
        date={2017},
     journal={Geom. Funct. Anal.},
      volume={27},
      number={3},
       pages={541\ndash 595},
}

\bib{Florea17}{article}{
      author={Florea, A.},
       title={Improving the error term in the mean value of {$L(\frac12,\chi)$}
  in the hyperelliptic ensemble},
        date={2017},
     journal={Int. Math. Res. Not. IMRN},
      number={20},
       pages={6119\ndash 6148},
}

\bib{Florea17-1}{article}{
      author={Florea, A.},
       title={The second and third moment of {$L(1/2,\chi)$} in the
  hyperelliptic ensemble},
        date={2017},
     journal={Forum Math.},
      volume={29},
      number={4},
       pages={873\ndash 892},
}

\bib{Gao2021-3}{article}{
      author={Gao, P.},
       title={Sharp lower bounds for moments of quadratic {D}irichlet
  {$L$}-functions},
        date={Preprint},
        note={arXiv:2102.04027},
}

\bib{Gao2021-2}{article}{
      author={Gao, P.},
       title={Sharp upper bounds for moments of quadratic {D}irichlet
  {$L$}-functions},
        date={Preprint},
        note={arXiv:2101.08483},
}

\bib{G&Zhao12}{article}{
      author={Gao, P.},
      author={Zhao, L.},
       title={Moments of quadratic {D}irichlet {$L$}-functions over function
  fields},
        date={2023},
     journal={Finite Fields Appl.},
      volume={85},
        note={Paper no. 102113},
}

\bib{G&Zhao2024-3}{article}{
      author={Gao, P.},
      author={Zhao, L.},
       title={Shifted moments of quadratic {D}irichlet {$L$}-functions},
        date={to appear},
     journal={Canad. Math. Bull.},
        note={arXiv:2406.18024},
}

\bib{Harper}{article}{
      author={Harper, A.~J.},
       title={Sharp conditional bounds for moments of the {R}iemann zeta
  function},
        date={Preprint},
        note={arXiv:1305.4618},
}

\bib{Jung}{article}{
      author={Jung, H.},
       title={Note on the mean value of {$L(\frac12,\chi)$} in the
  hyperelliptic ensemble},
        date={2013},
        ISSN={0022-314X},
     journal={J. Number Theory},
      volume={133},
      number={8},
       pages={2706\ndash 2714},
}

\bib{Jutila}{article}{
      author={Jutila, M.},
       title={On the mean value of ${L}(1/2,\chi)$ for real characters},
        date={1981},
     journal={Analysis},
      volume={1},
      number={2},
       pages={149\ndash 161},
}

\bib{KS1}{article}{
      author={Katz, N.},
      author={Sarnak, P.},
       title={Zeros of zeta functions and symmetries},
        date={1999},
     journal={Bull. Amer. Math. Soc.},
      volume={36},
      number={1},
       pages={1\ndash  26},
}

\bib{K&S}{book}{
      author={Katz, N.},
      author={Sarnak, P.},
       title={Random matrices, {F}robenius eigenvalues, and monodromy},
      series={American Mathematical Society Colloquium Publications},
   publisher={American Mathematical Society},
     address={Providence},
        date={1999},
      volume={45},
}

\bib{Keating-Snaith02}{article}{
      author={Keating, J.~P.},
      author={Snaith, N.~C.},
       title={Random matrix theory and {$L$}-functions at {$s=1/2$}},
        date={2000},
     journal={Comm. Math. Phys.},
      volume={214},
      number={1},
       pages={91\ndash 110},
}

\bib{Kou}{article}{
      author={Koukoulopoulos, D.},
       title={Pretentious multiplicative functions and the prime number theorem
  for arithmetic progressions},
        date={2013},
     journal={Compos. Math.},
      volume={149},
      number={7},
       pages={1129\ndash 1149},
}

\bib{Li2024}{article}{
      author={Li, Xiannan},
       title={Moments of quadratic twists of modular {$L$}-functions},
        date={2024},
     journal={Invent. Math.},
      volume={237},
      number={2},
       pages={697\ndash 733},
}

\bib{NgSheWon}{article}{
      author={Ng, N.},
      author={Shen, Q.},
      author={Wong, P.-J.},
       title={Shifted moments of the {R}iemann zeta function},
        date={2024},
     journal={Canad. J. Math.},
      volume={76},
      number={5},
       pages={1556\ndash 1586},
}

\bib{Radzi2011}{article}{
      author={Radziwi{\l \l}, M.},
       title={Large deviations in {S}elberg's central limit theorem},
        date={Preprint},
        note={arXiv:1108.5092},
}

\bib{Radziwill&Sound1}{article}{
      author={Radziwi{\l\l}, M.},
      author={Soundararajan, K.},
       title={Continuous lower bounds for moments of zeta and {$L$}-functions},
        date={2013},
     journal={Mathematika},
      volume={59},
      number={1},
       pages={119\ndash 128},
}

\bib{Rosen02}{book}{
      author={Rosen, M.},
       title={{N}umber {T}heory in {F}unction {F}ields},
      series={Graduate Texts in Mathematics},
   publisher={Springer-Verlag, New York},
        date={2002},
      volume={210},
}

\bib{R&Sound1}{article}{
      author={Rudnick, Z.},
      author={Soundararajan, K.},
       title={Lower bounds for moments of {$L$}-functions: symplectic and
  orthogonal examples},
     journal={in: Multiple {D}irichlet {S}eries, {A}utomorphic {F}orms, and
  {A}nalytic {N}umber {T}heory, 293--303, Proc. Sympos. Pure Math.},
      volume={75},
       pages={Amer. Math. Soc., Providence, RI, 2006.},
}

\bib{R&Sound}{article}{
      author={Rudnick, Z.},
      author={Soundararajan, K.},
       title={Lower bounds for moments of {$L$}-functions},
        date={2005},
     journal={Proc. Natl. Acad. Sci. USA},
      volume={102},
      number={19},
       pages={6837\ndash 6838},
}

\bib{Shen}{article}{
      author={Shen, Q.},
       title={The fourth moment of quadratic {D}irichlet {$L$}-functions},
        date={2021},
     journal={Math. Z.},
      volume={298},
      number={1-2},
       pages={713\ndash 745},
}

\bib{ShenStucky}{article}{
      author={Shen, Q.},
      author={Stucky, J.},
       title={The fourth moment of quadratic {D}irichlet {$L$}-functions {II}},
        date={Preprint},
        note={arXiv:2402.01497},
}

\bib{Sound2009}{article}{
      author={Soundararajan, K.},
       title={Moments of the {R}iemann zeta function},
        date={2009},
     journal={Ann. of Math. (2)},
      volume={170},
      number={2},
       pages={981\ndash 993},
}

\bib{S&Y}{article}{
      author={Soundararajan, K.},
      author={Young, M.~P.},
       title={The second moment of quadratic twists of modular
  {$L$}-functions},
        date={2010},
     journal={J. Eur. Math. Soc. (JEMS)},
      volume={12},
      number={5},
       pages={1097\ndash 1116},
}

\bib{Szab}{article}{
      author={Szab\'o, B.},
       title={High moments of theta functions and character sums},
        date={2024},
     journal={Mathematika},
      volume={70},
      number={2},
       pages={Paper No. e12242, 37 pp.},
}

\bib{Weil}{book}{
      author={Weil, A.},
       title={Sur les courbes alg\'{e}briques et les vari\'{e}t\'{e}s qui s'en
  d\'{e}duisent},
      series={Publ. Inst. Math. Univ. Strasbourg},
   publisher={Hermann et Cie., Paris},
        date={1948},
      volume={7},
}

\end{biblist}
\end{bibdiv}
\bibliographystyle{amsxport}

\end{document}